\documentclass{amsart}

\usepackage{color}

\calclayout

\def\vint_#1{\mathchoice%
        {\mathop{\kern 0.2em\vrule width 0.6em height 0.69678ex depth -0.58065ex
                \kern -0.8em \intop}\nolimits_{\kern -0.4em#1}}%
        {\mathop{\kern 0.1em\vrule width 0.5em height 0.69678ex depth -0.60387ex
                \kern -0.6em \intop}\nolimits_{#1}}%
        {\mathop{\kern 0.1em\vrule width 0.5em height 0.69678ex
            depth -0.60387ex
                \kern -0.6em \intop}\nolimits_{#1}}%
        {\mathop{\kern 0.1em\vrule width 0.5em height 0.69678ex depth -0.60387ex
                \kern -0.6em \intop}\nolimits_{#1}}}
\def\vintslides_#1{\mathchoice%
        {\mathop{\kern 0.1em\vrule width 0.5em height 0.697ex depth -0.581ex
                \kern -0.6em \intop}\nolimits_{\kern -0.4em#1}}%
        {\mathop{\kern 0.1em\vrule width 0.3em height 0.697ex depth -0.604ex
                \kern -0.4em \intop}\nolimits_{#1}}%
        {\mathop{\kern 0.1em\vrule width 0.3em height 0.697ex depth -0.604ex
                \kern -0.4em \intop}\nolimits_{#1}}%
        {\mathop{\kern 0.1em\vrule width 0.3em height 0.697ex depth -0.604ex
                \kern -0.4em \intop}\nolimits_{#1}}}

\newcommand{\aveint}[2]{\mathchoice%
        {\mathop{\kern 0.2em\vrule width 0.6em height 0.69678ex depth -0.58065ex
                \kern -0.8em \intop}\nolimits_{\kern -0.45em#1}^{#2}}%
        {\mathop{\kern 0.1em\vrule width 0.5em height 0.69678ex depth -0.60387ex
                \kern -0.6em \intop}\nolimits_{#1}^{#2}}%
        {\mathop{\kern 0.1em\vrule width 0.5em height 0.69678ex depth -0.60387ex
                \kern -0.6em \intop}\nolimits_{#1}^{#2}}%
        {\mathop{\kern 0.1em\vrule width 0.5em height 0.69678ex depth -0.60387ex
                \kern -0.6em \intop}\nolimits_{#1}^{#2}}}

\usepackage{hyperref}
\hypersetup{breaklinks=true}

\numberwithin{equation}{section}
\theoremstyle{plain}
\newtheorem{theorem}[equation]{Theorem}

\newtheorem{lemma}[equation]{Lemma}

\theoremstyle{remark}
\newtheorem{remark}[equation]{Remark}

\theoremstyle{definition}
\newtheorem{definition}[equation]{Definition}

\newtheorem*{question*}{Question}

\usepackage{datetime} 

\usepackage{graphicx}

\usepackage{url}

\title[Self-improvement and connectivity]{Alternative proof of Keith-Zhong self-improvement and connectivity}

\author[S.\! Eriksson-Bique]{Sylvester Eriksson-Bique}
\address[S.E.-B.]{Department of mathematics, UCLA, 520 Portola Plaza, Los Angeles CA 90095, USA}
\email{syerikss@math.ucla.edu}

\newcounter{prob}
\newcommand{\Z}{\ensuremath{\mathbb{Z}}}
\newcommand{\N}{\ensuremath{\mathbb{N}}}
\newcommand{\R}{\ensuremath{\mathbb{R}}}

\newcommand{\M}{{\mathcal M}}

\newcommand{\LIP}{\ensuremath{\ \mathrm{LIP\ }}}

\newcommand{\Lip}{\ensuremath{\mathrm{Lip\ }}}

\newcommand{\defeq}{\mathrel{\mathop:}=}

\newcommand{\len}{\ensuremath{\mathrm{Len}}}

\def\XXint#1#2#3{{\setbox0=\hbox{$#1{#2#3}{\int}$ }
\vcenter{\hbox{$#2#3$ }}\kern-.58\wd0}}

\newcommand{\co}{\mskip0.5mu\colon\thinspace}   

\begin{document}

\maketitle

\begin{abstract}
\noindent We find a new proof for the celebrated theorem of Keith and Zhong that a $(1,p)$-Poincar\'e inequality self-improves to a $(1,p-\epsilon)$-Poincar\'e inequality. The paper consists of a novel characterization of Poincar\'e inequalities and then uses it to give an entirely new proof which is closely related to Muckenhoupt-weights. This new characterization, and the alternative proof, demonstrate a formal similarity between Muckenhoupt-weights and Poincar\'e inequalities. The proofs we give are short and somewhat more direct. With them we can give the first completely transparent bounds for the quantity of self-improvement and the constants involved. We observe that the quantity of self-improvement is, for large $p$, directly proportional to $p$, and inversely proportional to a power of the doubling constant and the constant in the Poincar\'e inequality. The proofs can be localized and thus we obtain more transparent proofs of the self-improvement of local Poincar\'e inequalities. \\ \\


\end{abstract}

\tableofcontents

\section{Introduction}

\subsection{Self-improvement of Poincar\'e inequalities}

Our goal is two-fold. On the one hand, we wish to reprove a result by Keith and Zhong on the self-improvement of Poincar\'e inequalities \cite{keith2008poincare}, and to give explicit bounds for the quantity of self-improvement. Prior to Keith's and Zhong's result it was common to assume a $(1,q)$-Poincar\'e inequality for some $q<p$ when proving statements involving functions in the Sobolev space with an exponent $p$. The result of Keith and Zhong replaces this assumption with a more natural assumption of a $(1,p)$-Poincar\'e inequality, and thus is widely applied in the study of analysis on metric measure spaces. Despite its significance, its proof has remained somewhat myserious to many outside of a small community of experts. In order to remedy this situation, we aim to give a more direct and transparent proof, that is based on new ideas of iteration and curve fragments. These ideas may become useful in studying other self-improvement phenomena as well. 

On the other hand, our goal is to draw attention to an intimate connection between the theory of Muckenhoupt-weights (see \cite{stein2016harmonic}) and Poincar\'e inequalities. It is well-known, that the results of self-improvement for Muckenhoupt-weights and Poincar\'e inequalities bear striking similarity. However, that this similarity extends to the level of proofs and definitions is surprising. When the underlying metric space is $X=\R$, Muckenhoupt weights coincide with those doubling measures permitting Poincar\'e inequalities \cite{bjorn2006admissible}. In a general metric space the question is much more subtle, but we describe a sense in which a Poncar\'e inequality can be characterized by a Muckenhoupt-type condition ``along some curves''.

To state the result, we will need the following terminology. For simplicity, we will consistently work with proper metric measure spaces $(X,d,\mu)$ equipped with locally finite measures $\mu$ such that $0<\mu(B(x,r))<\infty$ for all open balls $B(x,r) \subset X$. 

\begin{definition}
A proper metric measure space $(X,d,\mu)$ equipped with a Radon measure $\mu$ is said to be {\sc $D$-doubling} if for all $0<r$ and any $x \in X$ we have 
\begin{equation}
\frac{\mu(B(x,2r))}{\mu(B(x,r))} \leq D.
\end{equation}
We say that $(X,d,\mu)$ is {\sc $D$-doubling up to scale $r_0$} if the same holds for all $r \in (0,r_0)$.
\end{definition}

The average of a measurable function $f \co X \to \R$ on a metric measure space $(X,d,\mu)$ over a measurable set $A$, with $0<\mu(A)<\infty$, is denoted by $$f_A\defeq \vint_A f ~ d\mu \defeq \frac{1}{\mu(A)} \int_A f ~d\mu,$$ when it makes sense, and it's local (upper) Lipschitz constant is defined as $$\Lip f(x) \defeq \limsup_{y \to x, y \neq x} \frac{|f(x)-f(y)|}{d(x,y)}.$$
If $B=B(x,r)$ is a ball, we denote $CB=B(x,Cr)$ (despite the ambiguity that a ball as a set may not be uniquely defined by a center and a radius).

\begin{definition} Let $1 \leq p < \infty$ be given.
A proper metric measure space $(X,d,\mu)$ with a Radon measure $\mu$ and $\text{supp}(\mu)=X$ is said to satisfy a {\sc $(1,p)$-Poincar\'e inequality (with constants $(C,C_{PI})$)} if for all Lipschitz functions $f$ and all $x \in X, 0<r$ we have for $B=B(x,r)$
$$\vint_B |f-f_B| ~d\mu \leq C_{PI}r  \left(\vint_{CB} (\Lip f)^p ~d\mu \right)^{\frac{1}{p}}.$$
We say that $(X,d,\mu)$ satisfies a {\sc $(1,p)$-Poincar\'e inequality (with constants $(C,C_{PI})$) up to scale $r_0>0$} if the same holds for all $r \in (0,r_0)$. If $X$ is $D$-doubling and satisfies a $(1,p)$-Poincar\'e inequality, then it is called a {\sc PI-space}.
\end{definition} 
This inequality could be expressed in different generalities, but we choose this simple expression as it is sufficient. For a detailed discussion of these issues we refer to \cite{keith2003modulus, hajlasz1995sobolev, Heinonen2000}. 

By an application of H\"older's inequality, we can see that for smaller $p$ the $(1,p)$-Poincar\'e inequality becomes stronger. Thus, the following theorem of Keith and Zhong is called a self-improvement result. For a more detailed discussion of the background 
we refer to \cite{keith2008poincare, bjorn2011nonlinear}. For the original proof, see \cite{keith2008poincare}, or its presentation in \cite{heinonen2015sobolev}.

\begin{theorem}[Keith-Zhong \cite{keith2008poincare}]\label{thm:maintheorem} Assume $p>1$. Let $(X,d,\mu)$ be a proper $D$-doubling metric measure space with a $(1,p)$-Poincar\'e inequality with constants $(C,C_{PI})$. There exists a positive constant $\epsilon(D,p,C_{PI})>0$  such that for\- any $\epsilon \in (0,\epsilon(D,p,C_{PI}))$ the space admits a $(1,p-\epsilon)$-Poincar\'e inequality with constants $C'=C'(D,C_{PI},\epsilon,C)$,  $C_{PI}'=C_{PI}'(D,p,C_{PI},\epsilon,C)$. 
\end{theorem}

Our proof gives the following bound, which shows that the quantity of self-improvement is independent of the inflation factor $C$,
$$\epsilon(D,p,C_{PI}) \geq \frac{p}{(2^{13p+3}C_{PI}^p D^{3p+4})^{\frac{1}{p-1}}}.$$

Letting $p \to \infty$, we obtain the asymptotic estimate for the improvement $\frac{p}{2^{12}C_{PI}D^{3}}$. This is, naturally, not a tight bound. This estimate means that for larger $p$ the improvement in Keith-Zhong becomes larger, and in fact is linearly proportional to it for large $p$. We remark, that sharp bounds for the self-improvement of Muckenhoupt-weights have been studied in \cite{hytonen2012sharp}, as well as the references mentioned therein. 

Another reproof has been concurrently developed by other authors in \cite{kinnunen2017maximal}. Their methods yield more general insights into self-improvement phenomena, while this write up is restricted to classical Poincar\'e inequalities. Also, a careful examination of their paper seems to lead to similar bounds for the self-improvement.

We would also like to mention the recent unpublished work of Luk\'a\v{s} Mal\'y on types of Lorentz-Poincar\'e inequalities without self-improvement, and general conditions for self-improvement for various types of Poincar\'e inequalities.  

\subsection{Proof techniques and characterizations of Poincar\'e inequalities}

We were motivated to re-investigate the beautiful and insightful proof of the Keith-Zhong result  \cite{keith2008poincare} for a few reasons. Firstly, the original proof is somewhat non-intuitive. It proceeds by an abstract argument estimating distributions of certain maximal functions, where the relationships between different estimates is only revealed at the very end. This makes the argument somewhat indirect. As a consequence, extracting bounds from their proof seems very complicated. This was done in \cite{heinonen2015sobolev}, but the bounds seem to deteriorate for large exponents $p$. The bounds we obtain below are much sharper. 

On the other hand, we have worked on more general applications of ``self-improvement''-type methods, where much of the machinery of the original proof of Keith and Zhong become unnecessary \cite{sylvester:poincare}. Our goal is to understand whether the framework of \cite{sylvester:poincare} could be used to provide an easier proof of the Keith-Zhong result. This framework is based on tools such as iteration and the idea of ``refilling'' curves. However, to achieve this goal we need new techniques, because the paper in \cite{sylvester:poincare} does not give sharp characterizations of Poincar\'e inequalities. More precisely, while those results are sharp in general, for several classes of spaces better results can be obtained, and thus we needed to develop an understanding of different characterizations.

These characterizations come in the flavor of Muckenhoupt-type conditions. Thus, an additional motivation of this paper is to study the formal similarity between Poincar\'e inequalities and Muckenhoupt-weights. This similarity was alluded to in our prior paper \cite{sylvester:poincare}, but we wish to make this formal analogy more precise. In the process, we obtain a new characterization of Poincar\'e-inequalities that clarifies the dependence of the exponent. This relationship to Muckenhoupt-weights has been previously observed in \cite{bjorn2006admissible} as a way of characterizing measures on $\R$ which admit Poincar\'e inequalities. Thus, our results can be thought of as weaker and higher dimensional analogues of such characterizations. 



\begin{theorem}\label{thm:classification} For a proper metric measure space $(X,d,\mu)$ which is $D$-doubling the following conditions are equivalent.

\begin{enumerate}
\item[$PI_p:$] $X$ satisfies a $(1,p)$-Poincar\'e inequality.
\item[$PtPI_p:$] $X$ satisfies a pointwise Poincar\'e inequality: There are constants $(C,C_{PPI})$ such that for every continuous $f$ and any upper gradient $g$ for $f$ and all $x,y \in X$  with $d(x,y)=r$ the following estimate holds:
$$|f(x)-f(y)| \leq C_{PPI}d(x,y)\left(\M_{p,Cr}g(x)+\M_{p,Cr}g(y)\right).$$ 
\item[$A_pC:$] $X$ is $A_p$-connected: There are constants $(C,C_{A})$ such that for every non-negative lower semi-continuous $g$, and any $x,y \in X$ with $d(x,y)=r$, there is a Lipschitz curve $\gamma$ connecting $x$ to $y$ with $\len(\gamma) \leq Cd(x,y)$ and
$$\int_{\gamma} g ~ds \leq C_{A} d(x,y)\left( \M_{p,Cr}g(x) + \M_{p,Cr}g(y) \right).$$
\end{enumerate}

The constants denoted $C$ with or without subscripts in the various statements can be different and depend quantitatively on each other.
\end{theorem}

Recall, that for a locally integrable and measurable function $f \in L^p_{loc}$ we define
$$\M_{p,s} f(x) \defeq \sup_{r \in (0,s)} \Bigg( \vint_{B(x,r)} f^p ~d\mu \Bigg)^\frac{1}{p}.$$

This result is closely related to a lemma by Heinonen and Koskela \cite[Lemma 5.1]{heinonen1998quasiconformal}. The novelty is on the new notion of $A_p$-connectivity that arises. It reduces the problem of proving a Poincar\'e inequality to finding a single curve, with controlled length and integral. However, the difficulty is to do this for an arbitrary pair of points and every function. In a sense, this notion of connectivity is nothing other than a  reformulated modulus estimate involving Riesz kernels from Keith \cite{keith2003modulus}. Here, the modulus condition is reformulated as a problem of finding curves with small integrals. This point, while present in some work, seems to not have been fully utilized, and doesn't appear explicitly in prior literature. Formally, the task of constructing a single curve is much easier than constructing ``thick'' curve families, which traditionally is involved in proving modulus estimates.

The task of constructing a curve can be done iteratively, which is the core idea in \cite{sylvester:poincare}, and is reformulated here. This idea involves both the notion of ``level'' and ``scale''. The iteration is started by constructing an initial curve, where some bad behavior occurs only on some small set. By replacing the portions in this bad set, we obtain a better curve than initially expected. The replacing is done at a smaller scale. The badness corresponds to the size of $\M_{p,s}g$, where $s$ is the scale we are working at. This size of $\M_{p,s}g$ is also referred to as a level, and will be a definite amount larger than the initial level. If a ``good'' level can be chosen, such that the size of the next scale is small enough compared to it, then we can obtain an absorbable lower order term. 

This idea of absorbing a term from a higher level and smaller scale is included indirectly in \cite{keith2008poincare}, and forms the core of many good-$\lambda$-type inequalities. The estimate, which involves the level in a scale invariant way, leads to the definition of an $\alpha$-function. This function describes the connectivity of the space and naturally encodes the iteration procedure. A similar function appear is \cite{lerner2007}, and our terminology is motivated by theirs. In \cite{sylvester:poincare} the iteration is done differently. There, the desired curve is directly constructed via an infinite recursion and limiting process, where at each step some ``gaps'' or undefined portions of the curve are refilled. Here, we can avoid both the use of ``gaps'' and the use of an infinite recursion. In a sense, the new function measures connectivity at various levels in a scale invariant way. 



Finally, we remark, that our methods are local, and thus we obtain the following transparent local version of self-improvement. 

\begin{theorem}\label{thm:maintheoremlocal} Assume $p>1$. Let $(X,d,\mu)$ be a proper metric measure space, which is $D$-doubling up to scale $r_D$ with a $(1,p)$-Poincar\'e inequality with constants $(C,C_{PI})$ up to scale $r_{PI}$. There exists a $\epsilon(D,p,C_{PI})>0$  such that for any $\epsilon \in (0,\epsilon(D,p,C_{PI}))$ the space admits a $(1,p-\epsilon)$-Poincar\'e inequality with constants $C'=C'(D,C_{PI},\epsilon,C), C_{PI}'=C_{PI}'(D,p,C_{PI},\epsilon,C)$ up to scale 
$$r_0 \leq \min\bigg\{\frac{r_{PI}}{4}, \frac{r_D}{20C}\bigg\}.$$ 
\end{theorem}

We end this introduction with an intuitive, and informal, reason for Theorem \ref{thm:maintheorem} to hold true. This intuition is abstractly present in the proofs of this paper. Given a function $g$, the $A_p$-connectivity from Theorem \ref{thm:classification} implies the existence of curves $\gamma$ connecting $x$ and $y$ with $d(x,y) = r$ a bound of the form
$$\int_{\gamma} g ~ds \leq C_{A} d(x,y)\left( \M_{p,Cr}g(x) + \M_{p,Cr}g(y) \right).$$
However, this bound is not optimal for all $g$. In the case where $g$ is supported on a very small set, or is highly concentrated, then a much better curve can be obtained. Namely, for any $\delta \in (0,1)$, there is an $\epsilon \in (0,1)$ such that if $E = \text{supp}(g)$ and $\M_{p,Cr}1_E(x) + \M_{p,Cr}1_E(y) < \epsilon $, then in fact we could get a bound roughly of the form
$$\int_{\gamma} g ~ds \leq \delta d(x,y)\left( \M_{p,Cr}g(x) + \M_{p,Cr}g(y) \right).$$
Here, the constant $C_A$ can be replaced with the much smaller $\delta$. Thus, highly concentrated ``obstacle'' functions $g$ are in fact easier to avoid. Quantifying this leads to the self-improvement phenomenon, since if $g$ is not highly concentrated, then the $L^q$ and $L^p$-norms become comparable. On the other hand, if $g$ is highly concentrated, then the previous sketch of an argument shows that the curve integrals are much smaller than expected.  \\ \\

\paragraph{\textbf{Acknowledgments}:} I thank my adviser Professor Bruce Kleiner for discussing similar topics, especially in relation to the previous paper \cite{sylvester:poincare}. I also thank Professor Juha Kinnunen and Antti V\"ah\"akangas for discussing their related work in \cite{kinnunen2017maximal} and giving feedback on the presentation of the current paper, and for presenting many interesting problems related to this work. I also thank Nageswari Shanmugalingam for encouraging us to rethink the proofs from an earlier version, which improved the presentation. Finally, the paper has benefited from a careful reading by the referee and his many corrections and comments. This research has been supported by NSF graduate fellowship DGE-1342536 and NSF grant DMS-1704215.

\section{Preliminary lemmas}\label{sec:prelim}
Throughout this paper we will assume that $(X,d,\mu)$ is a proper metric measure space equipped with a Radon measure $\mu$.

By a curve $\gamma \co I \to X$ we mean a continuous function whose domain $I \subset \R$ is compact. The length of an interval $I$ is denoted $|I|$. The length of a curve is defined as
\begin{equation}\label{eq:lendef}\len(\gamma) \defeq \sup_{x_1 \leq \dots \leq x_n \in I} \sum_{i =1}^{n-1} d(\gamma(x_{i+1}),\gamma(x_i)).\end{equation}
A curve $\gamma$ is called \emph{rectifiable} if $\len(\gamma)<\infty$. Most of the time we will focus on \emph{Lipschitz curves}, i.e. those for which there exists a $L \in (0,\infty)$ such that $d(\gamma(a),\gamma(b)) \leq L|b-a|$ for any $a,b \in I$. The smallest $L$ for which this inequality is satisfied is also called the Lipschitz constant of $\gamma$ and is denoted $\LIP(\gamma)$. If $\gamma$ is assumed to be Lipschitz, we have $\len(\gamma) \leq \LIP(\gamma)|I|.$ In fact, any curve can be reparametrized by length as $\gamma^* \co [0,\len(\gamma)] \to X$. This makes the curve $1$-Lipschitz, and such that $\len(\gamma|_{[a,b]}) = |b-a|$ whenever $0<a<b<\len(\gamma)$ \cite{ambrosio2008gradient}. 

For rectifiable curves one can define a curve integral according to \cite{ambrosio2008gradient}, and which is defined for any bounded/signed Borel function. In fact, if $\gamma$ is a rectifiable curve, and $\gamma^* \co [0,\len(\gamma)] \to X$ is its length-reparametrization, the integral can be defined as
$$\int_\gamma g ~ds \defeq \int_0^{\len(\gamma)} g(\gamma^*(t))~dt,$$
when the right-hand side makes sense.

A metric space $(X,d)$ is called \emph{($L-$)quasiconvex} if for every $x,y \in X$, there exists a rectifiable curve $\gamma$ connecting $x$ to $y$ with $\len(\gamma) \leq L d(x,y)$. A space that is $1$-quasiconvex is called geodesic. We recall, that a curve $\gamma \co I \to X$ is said to connect a pair of points $x,y$ if $\gamma(\min(I))=x,\gamma(\max(I))=y$. 

If $f$ is a continuous function on $X$, we call a non-negative Borel function $g$ an \emph{upper gradient} for $f$ if for every $x,y \in X$, and any rectifiable curve $\gamma$ connecting $x$ to $y$ we have
$$|f(x)-f(y)| \leq \int_\gamma g ~ds.$$
This terminology is due to Heinonen and Koskela \cite{heinonen1998quasiconformal}.

We define the localized Hardy-Littlewood maximal functions with exponent $p \in [1,\infty)$ as
$$\M_{p,s} f(x) = \sup_{r \in (0,s]} \left(\vint_{B(x,r)} f^p ~d\mu\right)^\frac{1}{p},$$
which makes sense for any non-negative measurable $f$. The non-localized version is simply 
$$\M_{p} f(x) = \sup_{0<r} \left(\vint_{B(x,r)} f^p ~d\mu\right)^\frac{1}{p}.$$
If $p=1$ we will drop the subscript. Finally, if $A \subset X$, then we denote by $1_A$ the characteristic function, or indicator function, of the set $A$.

We have the following weak $L^1$-distributional inequality. Its proof is contained in \cite{stein2016harmonic}.

\begin{theorem}(Maximal function estimate) Let $(X,d,\mu)$ be a $D$-measure doubling metric measure space and $s>0$ and $B(x,r) \subset X$ arbitrary, then for any $p \in [1,\infty)$ and any non-negative $ f \in L^1$ and $\lambda>0$ we have
$$\mu\left(\{ \M_{p,s} f > \lambda \}\cap B(x,r)\right) \leq D^3\frac{\Vert f 1_{B(x,r+s)} \Vert_{L^p}^p}{\lambda^p}.$$
\end{theorem}

We also need a different type of Maximal function estimate, whose proof is similar to the previous theorem, but with an additional observation. It is a ``multi-scale'' version of the previous inequality.

\begin{lemma}\label{lem:maxmax} (Max-max estimate) Let $(X,d,\mu)$ be $D$-measure doubling and $r,s>0$, $x \in X$ arbitrary. If $f$ is any non-negative measurable function and $E_{\lambda,p,s}  = \{z | \M_{p,s} f (z) > \lambda \}$, then we have
\begin{equation}
\M_{r}1_{E_{\lambda,p,s}}(x) \leq \frac{D^4 \left(\M_{p,s+r} f(x)\right)^p}{\lambda^p}.
\end{equation}
\end{lemma}

\begin{proof} Fix $x$ and $t \in (0,r)$ be arbitrary. Without loss of generality, assume $\infty > \lambda^p > D^4 \left(\M_{p,s+r} f(x) \right)^p$. Were this to fail, the estimate would become trivial (as the left hand side is bounded by $1$). We will estimate for the ball $B(x,t) \subset X$
$$\vint_{B(x,t)} 1_{E_{\lambda,p,s}} ~d\mu = \frac{\mu(E_{\lambda,p,s} \cap B(x,t))}{\mu(B(x,t))},$$
by the right hand side. The proof then follows by taking the supremum over $t\in (0,r]$.

Consider $A = B(x,t) \cap E_{\lambda,p,s}$. For every $z \in A$ there exists a ball $B(z,r_z)$ such that $r_z \in (0,s]$ and 
$$\vint_{B(z, r_z)} f^p ~d\mu > \lambda^p.$$
There are two cases. Either, for every $z \in A$ we have $r_z < t$, or there exists some such that $r_z \geq t$. If the latter case holds, then 
$$\vint_{B(x,r_z + t)} f^p \geq \frac{1}{D^2} \vint_{B(z,r_z)} f^p ~d\mu \geq \frac{\lambda^p}{D^2}.$$
But, from $r_z + t \leq s+r$ we would get $(\M_{p,s+r} f)^p \geq \frac{\lambda^p}{D^2}$, which gives a contradiction. So, for every $z$ we have $r_z < t$.

Using the 5-covering Lemma (see \cite{stein2016harmonic}), we obtain a collection of balls $\mathcal{B}=\{B(z_i, r_i)\}$ such that $B(z_i, r_i)$ are disjoint, so that $B(z_i, 5r_i)$ cover the set $A$, $z_i \in B(x,t)$, $r_i \in (0,\min(s,t)]$ and
$$\vint_{B(z_i,r_i)} f^p ~d\mu > \lambda^p.$$
Then, we get
\begin{eqnarray*}
\frac{\mu(E_{\lambda,p,s} \cap B(x,t))}{\mu(B(x,t))} &\leq& \frac{\sum_{B \in \mathcal{B}} \mu(5B)}{\mu(B(x,t))}\\
&\leq& D^3\frac{\sum_{B \in \mathcal{B}} \mu(B)}{\mu(B(x,t))}\\
&\leq& \frac{D^4}{\lambda^p} \frac{\sum_{B \in \mathcal{B}} \int_{B} f^p ~d\mu}{\mu(B(x,t+\min(s,t)))} \\
&\leq& \frac{D^4}{\lambda^p} \frac{\int_{B(x,t+\min(s,t))} f^p ~d\mu}{\mu(B(x,t+\min(s,t)))} \\
&\leq& \frac{D^4}{\lambda^p} \vint_{B(x,t+\min(s,t))} f^p ~d\mu \leq \frac{D^4  \left(\M_{p,s+r} f(x)\right)^p}{\lambda^p}.
\end{eqnarray*}
\end{proof}

\section{Proof of Self-improvement}\label{sec:connectivity}

We will use the following definition of $A_p$-connectivity.

\begin{definition}\label{def:apconn}Let $C,C_A>0, p \geq 1$. We say that a metric measure space $(X,d,\mu)$ is {\sc $A_p$-connected (with constants $(C,C_A)$)} if for every  $x,y \in X$ with $d(x,y)=r>0$, and every lower semi-continuous and non-negative $g \co X \to [0,\infty)$, there exists a $L>0$ and a Lipschitz curve $\gamma \co [0,L]\to X$ such that

\begin{enumerate}
\item $\gamma(0)=x$,
\item $\gamma(L)=y$,
\item $\len(\gamma) \leq Cr$ and
\item 
\begin{equation}\label{eq:apconnI}
\int_\gamma g \leq C_A r \left(\M_{p,Cr}g(x)+\M_{p,Cr}g(y)\right).
\end{equation}
\end{enumerate}
\end{definition}

We choose the term $A_p$-connected to draw an analogy to the definition of $A_p$-weights. Recall, that the class of $A_p$-weights is defined by $\mu \in A_p(\lambda)$, where $\lambda$ is Lebesgue measure on $\R^n$ and $d\mu = \omega ~d\lambda$, if one of the following equivalent conditions holds.

\begin{enumerate}
\item Maximal function bound: There is a constant $C>0$ such that for every $f \in L^p(\mu)$ we have
\begin{equation}\label{eq:muckmax}
\left( \int (\M f)^p ~d\mu \right)^{\frac{1}{p}} \leq C \left( \int f^p ~d\mu \right)^{\frac{1}{p}}.
\end{equation}
\item Integral bound: $\mu=\omega \lambda$, where $\omega, \omega^{1-p}$ are locally integrable and there is a $C>0$ such that for every ball $B=B(x,r)$
\begin{equation} \label{eq:muckint}
\left(\vint_B \omega ~d\lambda\right)\left(\vint_B \omega^{1-p} ~d\lambda \right)^{\frac{1}{p-1}} \leq  C.
\end{equation}

\item Average bound: For some $C>0$ and for any $f$ locally integrable and any ball $B=B(x,r)$
\begin{equation}\label{eq:muckav}
\vint_B f ~d\lambda \leq C \left(\frac{1}{\mu(B)}\int_B f^p  ~d\mu \right)^{\frac{1}{p}}.
\end{equation}
\end{enumerate}

Further, all of these imply that a version of quantitative absolute continuity holds. By this, we mean that there is a constant $C>0$ such that for all $B(x,r)$ and all $E \subset B(x,r)$ we have
\begin{equation}\label{eq:muckset}
\frac{\lambda(E)}{\lambda(B(x,r))} \leq C \left(\frac{\mu(E)}{\mu(B(x,r))}\right)^{\frac{1}{p}}.
\end{equation}

It is subtle, that this quantitative absolute continuity is \emph{not} equivalent to being an $A_p$-weight. In fact, by work in \cite{kurtz1982, lerner2007} the condition \eqref{eq:muckset} characterizes so called $A_{p,1}$-weights. While the $A_p$-conditions characterize boundedness of the Hardy-Littlewood maximal function $\M$ from $L^p$ to $L^p$, the $A_{p,1}$-condition characterizes boundedness from $L^p \to L^{p,\infty}$. It is known, that $A_p \subset A_{p,1}$ strictly. Further, the $A_{p,1}$-condition does not improve to $A_{q,1}$ for any $q<p$. 

Our definition of $A_p$-connected is analogous to the average bound \eqref{eq:muckav}. Namely, replace $\lambda$ by $\mathcal{H}^1|_\gamma$ and the right-hand side by a maximal function bound. The measure $\mathcal{H}^1_\gamma$  is the 1-dimensional Hausdorff measure on the image of $\gamma$. The formal difference is that the condition of $A_p$-connectivity additionally presumed the \emph{existence} of some curve $\gamma$ such that the estimate holds. In a sense, the $A_p$-connectivity corresponds to being an  ``$A_p$-weight'' with respect to one-dimensional Hausdorff measure on some curve. 

The condition \eqref{eq:muckset} is somewhat similar to the notion of fine connectivity in \cite{sylvester:poincare}. It would correspond to restricting functions $g$ in the definition of $A_p$-connectivity, with characteristic functions $g=1_E$. However, we do not need to use that definition here.

At the heart of our proof of self-improvement is the characterization of Poincar\'e inequalities in terms of $A_p$-connectivity. We first need some elementary lemmas. 

\begin{lemma}\label{lem:approxset} Let $E$ be a Borel set, $p \in [1,\infty)$ and $s>0$ a fixed scale parameter. Then, for every $\epsilon \in (0,1)$ there exists an open set $O$ such that $E \setminus \{x\} \subset O$, $x \not\in O$ and such that
\[\M_{s}1_O(x) \leq \M_{s}1_E(x)+\epsilon,\]
and
\[\M_{p,s}1_{O\setminus E}(x) \leq \epsilon.\]
\end{lemma}

\begin{proof}
 Fix $\epsilon>0$. By regularity of measure, for each $n \in \Z$ we can find open sets $O_{\epsilon,n}$ such that $E \cap (B(x,2^{1-n}s) \setminus B(x,2^{-n-1}s)) \subset O_{\epsilon,n}$, $O_{\epsilon,n} \subset B(x,2^{2-n}s) \setminus B(x,2^{-n-2}s)$ and $\mu(O_{\epsilon,n} \setminus E \cap B(x,2^{2-n}s)) \leq \epsilon^p 4^{-n-1} \mu(B(x,2^{-n-2}s))$. Define $O = \bigcup_{n=0}^\infty O_{\epsilon,n}$. It is clear that $x \not\in O$ Now, clearly
 $$E \cap B(x,s) \setminus \{x\}  = \bigcup_{n \in \Z} E \cap (B(x,2^{1-n}s) \setminus B(x,2^{-n-1}s)) \subset \bigcup_{n \in \Z} O_{\epsilon,n}=O.$$
 
 Also, for any $t \in (0,s]$ we have
 \begin{eqnarray*}
    \vint_{B(x,t)} 1_{O\setminus E} ~d\mu &\leq& \frac{1}{\mu(B(x,t))} \sum_{n \in \Z, 2^{-n-2}s<t} \mu(O_{\epsilon,n} \setminus E \cap B(x,2^{2-n}s)) \\
    &\leq& \epsilon^p,
 \end{eqnarray*}
 which gives the second estimate in the statement of the Lemma. Similarly, the first statement follows from the following estimate.
 \begin{eqnarray*}
    \vint_{B(x,t)} 1_O ~d\mu &\leq& \vint_{B(x,t)} 1_E + 1_{O\setminus E} ~d\mu = \M_s 1_E(x) + \frac{1}{\mu(B(x,t))} \sum_{n \in \Z, 2^{-n-2}s<t} \mu(O_{\epsilon,n} \setminus E \cap B(x,2^{2-n}s)) \\
    &\leq& \M_s 1_E(x)+\epsilon^p \leq \M_s 1_E(x)+\epsilon.
 \end{eqnarray*}

\end{proof}

\begin{lemma}\label{lem:approx} Let $g$ be a non-negative Borel function such that $\M_{p,s}g(x)<\infty$. Then, for every $\epsilon \in (0,1)$ there exists an lower semi-continuous function $\overline{g}(y)$ such that $\overline{g}(y) \geq g(y)$ for all $y \neq x$ such that
\[\M_{p,s}\overline{g}(x) \leq \M_{p,s}g(x)+\epsilon.\]
\end{lemma}

\begin{proof}
 Let $E_{k,\epsilon} = \{y \in X \setminus \{x\} | g(y) >\frac{k\epsilon}{2}\}$ for $k \geq 0$. Clearly, for $y \neq x$ we have 
 \[ g(y) \geq \sum_{k=0}^\infty \frac{\epsilon}{2} 1_{E_{k,\epsilon}} - \frac{\epsilon}{2}.\]
 By Lemma \ref{lem:approxset} we have sets $O_{k,\epsilon}$ such that $E_{k,\epsilon} \subset O_{k,\epsilon}$ and
\[\M_{p,s}1_{O_{k,\epsilon}\setminus E_{k,\epsilon}}(x) \leq \epsilon 2^{-k-2}.\]
 
 Finally, define
 $$g_\epsilon = \sum_{k=0}^\infty \frac{\epsilon}{2} 1_{O_{k,\epsilon}}.$$
 Now, it is easy to obtain that $g_\epsilon \geq g$, except possibly at $x$. Finally, we also have for all $y \neq x$ 
 
 \[0<g_\epsilon(y) - g(y) \leq \frac{\epsilon}{2} + \sum_{k=0}^\infty \frac{\epsilon}{2} 1_{O_{k,\epsilon} \setminus E_{k,\epsilon}}(y).\]
 Thus, from the triangle inequality, we can derive
 \begin{eqnarray*}
   \M_{p,s}g_\epsilon &\leq& \M_{p,s}g + \M_{p,s}(g_\epsilon-g) \leq \M_{p,s}g + \frac{\epsilon}{2} + \sum_{k=0}^\infty \M_{p,s} 1_{O_{k,\epsilon} \setminus E_{k,\epsilon}} \\
   &\leq & \M_{p,s}g + \epsilon.
 \end{eqnarray*}
 
\end{proof}

\begin{remark} \label{rmk:multiplepoints}
 If we have a finite set of points $z_1, \dots, z_n$, then we can choose the lower-semi-continuous approximant $\overline{g}$ so that $\overline{g}(y) \geq g(y)$ for all $y \neq z_1, \dots, z_n$ and $\M_{p,s}\overline{g}(z_i) \leq \M_{p,s}g(z_i)+\epsilon$ for all $i = 1, \dots, n$. Namely, apply the lemma to give functions $\overline{g}_{z_i}$ that satisfy the conclusion for $x=z_i$, and define $\overline{g}=\min_{i = 1, \dots, n}\overline{g}_{z_i}$. Similarly, for Lemma \ref{lem:approxset}, we can ensure $\M_{s}1_O(x) \leq \M_{s}1_E(x)+\epsilon$ and $z_i \not\in O$ simultaneously for a finite set of points $x=z_1, \dots, z_n$ by considering the intersection of open sets $O_{z_i}$ satisfying the conlusion for individual $x=z_i$.
\end{remark}

\begin{proof}[ Proof of Theorem \ref{thm:classification}]
That $PI_p \Leftrightarrow PtPI_p$ follows from a classical result, which is presented for example in \cite[Lemma 5.15]{heinonen1998quasiconformal} combined with \cite[Theorem 2]{keith2003modulus}\footnote{Keith's result is also needed, since Heinonen and Koskela \cite{heinonen1998quasiconformal} use a slightly different definition of a Poincar\'e inequality.}. 

Next, we show that $A_pC \Rightarrow PtPI_p$. Let $g$ be a measurable upper gradient of a continuous function $f$. Then, using Lemma \ref{lem:approx} and Remark \ref{rmk:multiplepoints} we can find a $\overline{g_\epsilon}$ which is lower semi-continuous, $\overline{g_\epsilon} \geq g$ (except possibly at $x,y$) and $$\lim_{\epsilon \to 0}\M_{p,Cr}\overline{g_\epsilon}(x) = \M_{p,Cr}g(x), \lim_{\epsilon \to 0}\M_{p,Cr}\overline{g_\epsilon}(y) = \M_{p,Cr}g(y).$$

Then, for any rectifiable curve $\gamma$ parametrized by length, connecting a pair of points $x,y \in X$, we have
$$|f(x)-f(y)| \leq \int_\gamma g ~ds \leq \int_{\gamma} \overline{g_\epsilon} ~ds.$$

So, infimizing over curves $\gamma$ gives
$$|f(x)-f(y)| \leq C_{A} d(x,y) \left( \M_{p,Cr}\overline{g_\epsilon (x)} + \M_{p,Cr}\overline{g_\epsilon} (y) \right),$$
and then letting $\epsilon$ tend to zero gives the desired conclusion.

It remains to show that $PtPI_p$ and $PI_p$ imply $A_pC$. Assume that $(X,d,\mu)$ satisfies a $(1,p)$-Poincar\'e inequality and $PtPI_p$, and let $g$ be an arbitrary non-negative lower semi-continuous function such that $g^p$ is locally integrable and fix $x,y \in X$. To fix constants, assume the Poincar\'e inequality in the form
$$\vint_B |f-f_B| d\mu \leq C_{PI}r  \left(\vint_{CB} (\Lip f)^p d\mu \right)^{\frac{1}{p}}$$
and the second condition as 
\begin{equation}\label{eq:PPI}
 |f(x)-f(y)| \leq C_{PPI}d(x,y)\left(\M_{p,Cr}g_f(x)+\M_{p,Cr}g_f(y)\right),
\end{equation}
if $g_f$ is an upper gradient for $f$.

We will construct $\gamma$ such that 
$$\len(\gamma) \leq 5C_{PPI}d(x,y)$$
and
$$\int_\gamma g ~ds \leq 4C_{PPI} d(x,y) \left(\M_{p,Cr}g(x)+\M_{p,Cr}g(y)\right).$$
Next, define for every $N>0$ and $\epsilon>0$ a function $g_{N,\epsilon} = \min(g+\epsilon,N)$. Then
\begin{equation} 
\M_{p,Cr}g_{N,\epsilon}(x)+\M_{p,Cr}g_{N,\epsilon}(y) \leq  \M_{p,Cr}g(x)+ \M_{p,Cr}g(y) + 2\epsilon. 
\end{equation}

Now, define for $z \in X$ the set $\Gamma_{x,z}$ as the set of all rectifiable curves starting at $x$ and ending at $z$. Further, define a function by
\begin{equation}
\mathcal{F}_{N,\epsilon}(z) = \inf_{\gamma \in \Gamma_{x,z}} \int_{\gamma} g_{N,\epsilon} ~ds.
\end{equation}
This function is bounded and continuous, since PI-spaces are $L$-quasiconvex for some $L=L(C_{PI},D)$  (see e.g. \cite[Theorem 4.32]{bjorn2011nonlinear}, or \cite[Appendix]{ChDiff99}). It is also easy to see that $g_{N,\epsilon}$ is an upper gradient for $\mathcal{F}_{N,\epsilon}$. Next, by the $PtPI_p$-condition we have
\begin{equation}
\left|\mathcal{F}_{N,\epsilon}(y) -\mathcal{F}_{N,\epsilon}(x)\right| \leq C_{PPI}d(x,y)\left(\M_{p,Cr}g(x)+ \M_{p,Cr}g(y)\right) + 2C_{PPI}d(x,y)\epsilon. 
\end{equation}

Thus, there is a curve $\gamma_{N,\epsilon}$ such that $\gamma_{N,\epsilon}$ connects $x$ to $y$ and
$$\int_{\gamma_{N,\epsilon}} g_{N,\epsilon} ~ds \leq C_{PPI}d(x,y)\left(\M_{p,Cr}g(x)+\M_{p,Cr}g(y)\right) + 3C_{PPI}d(x,y)\epsilon.$$

Assume now $\epsilon > \M_{p,Cr}g(x)+\M_{p,Cr}g(y)$ is arbitrary. Then since $g_{N,\epsilon} \geq \epsilon$, we get
\begin{eqnarray}
\epsilon \len(\gamma_{N,\epsilon}) &\leq& \int_{\gamma_{N,\epsilon}} g_{N,\epsilon}~ds \nonumber \\
&\leq& 2C_{PPI}d(x,y)\left(\M_{p,Cr}g(x)+\M_{p,Cr}g(y)\right) + 3C_{PPI}d(x,y)\epsilon \nonumber \\
&\leq& 5C_{PPI}d(x,y)\epsilon.
\end{eqnarray}

Thus $\len(\gamma_{N,\epsilon}) \leq 5C_{PPI}d(x,y)$. Assume that $\gamma_{N,\epsilon}$ are parametrized by length. Then, they are $1$-Lipschitz and the properness of $X$ allows us to apply Arzela-Ascoli, and to extract a subsequential limit curve $\gamma_\epsilon$ connecting $x$ to $y$. Up to reindexing, we can assume that the curve is the limit of the original sequence. 
Then for every $N$, using lower semi-continuity of curve integrals and the lower semi-continuity of $g_{N,\epsilon}$ (see \cite[Proposition 4]{keith2003modulus}, we get
\begin{eqnarray*}
\int_{\gamma_\epsilon} g_{N,\epsilon} ~ds &\leq& \liminf_{M \to \infty} \int_{\gamma_{M, \epsilon}} g_{N,\epsilon} ~ds \\
&\leq& \liminf_{M \to \infty} \int_{\gamma_{M,\epsilon}} g_{M,\epsilon} \\
&\leq& C_{PPI}d(x,y)\left(\M_{p,Cr}g(x)+\M_{p,Cr}g(y)\right) + 3C_{PPI}d(x,y)\epsilon.
\end{eqnarray*}

Now, letting $N \to \infty$ and using monotone convergence, we get
\begin{equation}
\int_{\gamma_{\epsilon}} g ~ds \leq C_{PPI}d(x,y)\left(\M_{p,Cr}g(x)+\M_{p,Cr}g(y)\right) + 3C_{PPI}d(x,y)\epsilon.
\end{equation}

Let $\epsilon \to \left(\M_{p,Cr}g(x)+\M_{p,Cr}g(y)\right)$. Then, by using Arzela-Ascoli again we obtain a sub-sequential limit $\gamma$ of $\gamma_{\epsilon}$. Finally, using lower semi-continuity of $g$ and the lower semi-continuity of curve integrals we get $\len(\gamma) \leq 5C_{PPI}d(x,y)$ and the desired estimate
$$\int_{\gamma} g ~ds \leq 4C_{PPI}d(x,y)\left(\M_{p,Cr}g(x)+\M_{p,Cr}g(y)\right).$$

We remark, that this final limiting process is only necessary if $\M_{p,Cr}g(x)+\M_{p,Cr}g(y)=0$. Otherwise, we could just set $\epsilon = \M_{p,Cr}g(x)+\M_{p,Cr}g(y).$

\end{proof}

Next, we present our proof of Keith-Zhong self-improvement. Some notation and ideas are similar to \cite{lerner2007}, where the authors show general self-improvement phenomena for Maximal-function estimates. There, a crucial role is played by a sub-multiplicative function $\alpha$. For us, the relevant quantity is the following.

Let $x,y \in X$ be given and denote $r=d(x,y)$. Define with
\begin{equation}\label{eq:obsdef}
\mathcal{E}^p_{x,y,\tau,C} \defeq \{~g \co X \to [0,1]~| ~g \text{ lower semi-continous } \M_{p,Cr}g (x) + \M_{p,Cr}g (y) < \tau~\}
\end{equation} 
the set of admissible obstacle functions. Denote by $\Gamma_{x,y}^C$ the set of rectifiable curves $\gamma$ parametrized by length on the interval $[0,\len(\gamma)]$ such that $\gamma(0)=x, \gamma(\len(\gamma))=y$ and $\len(\gamma) \leq Cd(x,y)$. Then define
\begin{equation}\label{alphaindex}
\alpha^p(C,\tau) \defeq\sup_{x,y \in X} \sup_{g \in \mathcal{E}^p_{x,y, \tau,C} } \inf_{\gamma \in \Gamma_{x,y}^C} \frac{1}{d(x,y)}\int_\gamma g ~ds.
\end{equation}

In a sense, $\alpha^p(C,\tau)$ measures how well a function $g$ with ``small'' size can block curves, for the worst scale $d(x,y)$ and worst pair of points $x,y \in X$. The additional constraint on $g$ to have values in $[0,1]$ is used to ensure that $\alpha^p$ is bounded. Namely, if $X$ is $L$-quasiconvex with $L \leq C$, then for every $p \in [1,\infty)$
$$\alpha^p(C,\tau) \leq L,$$
for all $\tau \in [0,1]$. In this case, one can estimate the infimum from above by an arbitrary curve $\gamma$ connecting $x,y$ with length $\len(\gamma) \leq Ld(x,y)$, and obtain
$$\frac{1}{d(x,y)}\int_\gamma g ~ds \leq \frac{Ld(x,y)}{d(x,y)} \leq L.$$

While initially non-intuitive, this expression is a way of condensing the $A_p$-connectivity property. 

\begin{lemma} \label{lem:apchar} Let $p \in [1,\infty)$. The space $X$ is $A_p$-connected with constants $(C,C_A)$ if and only if
$$\alpha^p(C,\tau) \leq C_A \tau,$$
for all $\tau \in [0,1]$.
\end{lemma}

\begin{proof}
 If $X$ is $A_p$-connected, then for any $g \in \mathcal{E}^p_{x,y,\tau,C}$, we have
 $$\inf_{\gamma \in \Gamma_{x,y}^C} \int_\gamma g ~ds \leq C_A\tau d(x,y),$$
 since $\M_{p,Cr}g (x) + \M_{p,Cr}g (y) < \tau$ by assumption. Now, dividing both sides by $d(x,y)$ and taking a supremum over $g \in \mathcal{E}^p_{x,y,\tau, C}$ and $x,y$ gives the desired inequality for $\alpha^p$.
 
 The converse direction is somewhat more involved, as the $g$ in the $A_p$-connectivity condition need not be bounded.  This can be resolved with a limiting argument which uses the completeness of $X$. Now, to verify $A_p$-connectivity, we need to fix arbitrary $x,y \in X$ and a lower semi-continuous non-negative $g$ and find a curve $\gamma \in \Gamma_{x,y}^C$ with
 $$\int_{\gamma} g ~ds \leq C_{A} r\left( \M_{p,Cr}g(x) + \M_{p,Cr}g(y) \right).$$
 
 First, let $g_{N} = \min\{g,N\}$, and $\overline{g_N} = g_N/(2N)$. Both of these functions are lower semi-continuous and $\overline{g_N}(z) \in [0,1]$ for every $z \in X$. Let $\tau_N = \left( \M_{p,Cr}g_N(x) + \M_{p,Cr}g_N(y) \right)/(2N)$. Since $g_N$ converges to $g$ and is a monotone sequence, it is not hard to see that $\lim_{N \to \infty} 2N\tau_N = \M_{p,Cr}g(x) + \M_{p,Cr}g(y)$. 
 
 Now, by linearity it is easy to see that
 $$ \M_{p,Cr}\overline{g_N}(x) + \M_{p,Cr}\overline{g_N}(y) = \tau_N \in [0,1].$$
 
 Thus, $\overline{g_N} \in \mathcal{E}^p_{x,y,\tau_N,C}$. Then, from the definition of $\alpha^p(C,\tau_N)$, and the estimate for $\alpha^p$, we find for every $\epsilon>0$ a curve $\gamma_{\epsilon,N} \in \Gamma_{x,y}^C$ such that
 $$\int_{\gamma_{\epsilon,N}} \overline{g_N} ~ds \leq (\alpha^p(C,\tau_N) + \epsilon/N) d(x,y) \leq (C_A \tau_N + \epsilon/N) d(x,y).$$
 
 Multiplying both sides by $N$, we obtain
 $$\int_{\gamma_{\epsilon,N}} \min\{g,N\} ~ds \leq ( 2 C_A \tau_N N + \epsilon) d(x,y).$$
 
 Since $\gamma_{\epsilon, N} \in \Gamma_{x,y}^C$, and due to Arzela-Ascoli the set $\Gamma_{x,y}^C$ is a compact family of curves (with respect to uniform convergence), we can find a subsequential limit $\gamma_{\epsilon}$ of $\gamma_{\epsilon, N}$ as $N \to \infty$. To simplify notation, reindex so that this is the original sequence.
 
 Then, using monotone convergence and the lower semi-continuity of curve integrals.
 \begin{align*}
    \int_{\gamma_\epsilon} g ~ds &\leq \lim_{M \to \infty} \int_{\gamma_{\epsilon}} g_M ~ds \leq  \lim_{M \to \infty} \liminf_{N \to \infty} \int_{\gamma_{\epsilon,N}} \min\{g,M\} ~ds \\
                            &\leq \lim_{M \to \infty} \liminf_{N \to \infty} \int_{\gamma_{\epsilon,N}} \min\{g,N\} ~ds \\
                            &\leq \lim_{M \to \infty} \liminf_{N \to \infty}(2C_A \tau_N N + \epsilon) d(x,y)\\
                            &\leq (C_A (\M_{p,Cr}g(x) + \M_{p,Cr}g(y)) + \epsilon) d(x,y).
 \end{align*}
 
 Finally, letting $\epsilon$ tend to zero and using the lower semi-continuity of curve integrals again, we obtain a limit curve $\gamma$ of some subsequence of $\gamma_\epsilon$ such that
 \begin{align*}
    \int_{\gamma} g ~ds \leq C_A d(x,y)(\M_{p,Cr}g(x) + \M_{p,Cr}g(y)),
 \end{align*}
 as required.
 
\end{proof}

There is a simple sub-linear estimate for $\alpha^p$.

\begin{lemma}\label{lem:sublin} Let $K \geq 1$ and $p \in [1,\infty)$. Then
$$\alpha^p(C,K\tau) \leq K \alpha^p(C,\tau).$$
\end{lemma}

\begin{proof}
 Let $g \in \mathcal{E}^p_{x,y,K\tau,C}$, then $\overline{g}_K = g/K \in \mathcal{E}^p_{x,y,\tau,C}$. Note, that if $K \geq 1$, then $\overline{g}$ still has values in $[0,1]$. In particular, for any $x,y$ and $\gamma$, we have
 $$\frac{1}{d(x,y)}\int_{\gamma} g ~ds = K \frac{1}{d(x,y)} \int_{\gamma} \overline{g}_K ~ds,$$
 and so
 \begin{align*}
    \alpha^p(C,K\tau) &= \sup_{x,y} \sup_{g \in \mathcal{E}^p_{x,y,K\tau,C}} \inf_{\gamma \in \Gamma_{x,y}^C} \frac{1}{d(x,y)}\int_{\gamma} g ~ds \\
		    &= K \sup_{x,y} \sup_{g \in \mathcal{E}^p_{x,y,K\tau,C}} \inf_{\gamma \in \Gamma_{x,y}^C} \frac{1}{d(x,y)}\int_{\gamma} \overline{g}_K ~ds \\
		    &\leq K \sup_{x,y} \sup_{h \in \mathcal{E}^p_{x,y,\tau,C}} \inf_{\gamma \in \Gamma_{x,y}^C} \frac{1}{d(x,y)}\int_{\gamma} h ~ds = K \alpha^p(C,\tau).\\
 \end{align*}
\end{proof}

Since in the following proof we are using $A_p$ connectivity to prove $A_q$-connectivity, we will explicate their connectivity constants with an additional subscript. That is $A_p$ connectivity will be assumed to hold with constants $(C,C_{A,p})$, and we will prove $A_q$-connectivity with different constants $(L, C_{A,q})$ and with $L \geq C$.

\begin{theorem} Assume $p>1$. If $(X,d,\mu)$ is $D$-doubling and $A_p$-connected (with constants $C,C_{A,p}$), then there is a $\epsilon(D,p,C_{A,p})$ such that $X$ is also $A_{q}$-connected for all $p-\epsilon(D,p,C_{A,p})<q<p$ with constants depending on $C,C_{A,p},p$ and $q$.
\end{theorem}


\begin{proof} 
Recall the definitions of $\alpha^q$ in \eqref{alphaindex}. As discussed above, it is sufficient to show that there is an $\epsilon(D,p,C_{A,p})$, such that if $p-\epsilon(D,p,C_{A,p})<q<p$, then the space is $A_q$-connected. Fix this $\epsilon(D,p,C_{A,p})$ to be determined, and any such exponent $q$. The $A_q$ connectivity can be reduced by Lemma \ref{lem:apchar} to showing that there are some constants $L$ and $C_{A,q}$ such that for all $\tau \in [0,1]$
$$\alpha^q(L,\tau) \leq C_{A,q} \tau.$$

This estimate is shown by proving that for $\delta \in (0,1)$ there are some $k \in \N, M \geq 2, S \geq 1$ and for all $L \geq \frac{C}{1-\delta}$ and $\tau \in [0,1]$ we have
\begin{equation}\label{eq:crucialest}
 \alpha^q(L,\tau)\leq S\tau + \delta \max_{i=1, \dots, k}M^{-i}\alpha^q(L,M^{i}\tau).
\end{equation}




If we have this estimate, then Lemma \ref{lem:sublin} gives
\begin{equation}\label{eq:whatwewant}
 \alpha^q(L,\tau)\leq  S\tau + \delta \alpha^q(L,\tau),
\end{equation}
and so
$$\alpha^q(L,\tau)\leq  \frac{S}{1-\delta}\tau,$$
which is the estimate we desire with $C_{A,q} = \frac{S}{1-\delta}$. Note, we used the fact that $\alpha^q$ is bounded since $X$ is $C$-quasiconvex and $L \geq C$. This follows from $A_p$-connectivity since $\Gamma_{x,y}^C$ is not empty for any $x,y\in X$, as otherwise $\alpha^p(C,\tau)$ would not be bounded. Thus, it remains to prove \eqref{eq:crucialest}. 

Before we prove this, we wish to give some intuition. A way to think of this estimate \eqref{eq:crucialest} is that, we would really like to prove for some $\Lambda,L > 1$, that $\alpha^q(L,\tau)\leq  S\tau + \delta \Lambda^{-1}\alpha^q(L,\Lambda \tau)$. Here $\alpha^q(L,\tau)$ corresponds to an estimate ``at level'' $\tau$, and we can estimate it from above by a small constant $\delta$ times a term at level $\Lambda \tau$ and a term of the desired form. This corresponds to an iteration at the level of curves, as we will soon see, since an initial curve $\gamma$ is constructed to almost avoid points where $\M_{q,Ld(x,y)}g(z) > \Lambda \tau$. It can not fully avoid this set, but replacing the portions in this set, if it is sufficiently small, gives a contribution of the form $\delta \Lambda^{-1}\alpha^q(L,\Lambda \tau)$. 

However, we do not know how to choose $\Lambda$ \emph{a priori} in a way independent of $g \in \mathcal{E}^q_{x,y,\tau,L}$. The proof instead shows that we can always find one ``level'' $\Lambda \tau = M^{i_0} \tau$ for some $i_0 =1, \dots, k$ where a desired quantity is sufficiently small (compared to $\delta$). This leads to the less intuitive estimate \eqref{eq:crucialest} involving the maximum. However, this suffices for our purposes. We note, that the curves used for $\alpha^q$ have length $Ld(x,y)$ due to the fact that the iteration necessarily increases the lengths slightly. However, $L = \frac{C}{1-\delta}$ will suffice for our purposes. It is chosen so that $C + \delta L = L$.

Next, fix any $M \geq 2$ and any $\delta \in (0,1)$. We will also fix $k \geq 1$ to be determined later and $L = \frac{C}{1-\delta}$. In order to estimate $\alpha^q(L,\tau)$ we are taking a supremum over pairs of points and functions. Thus, let $x,y \in X$, $\tau \in [0,1]$ and let
\begin{equation}\label{eq:gdef}
g \in \mathcal{E}^q_{x,y, \tau,L} 
\end{equation}
be arbitrary. Denote $d(x,y)=r$. By adding a small constant to $g$ we can assume $\left(\M_{q,Lr} g(x)\right)^q + \left(\M_{q,Lr} g(y)\right)^q>0$. First define $F_i=\{z | \M_{q,\delta Lr} g > M^i \tau /2\}$, for which Lemma \ref{lem:maxmax} gives for $z=x,y$
$$\M_{Cr}1_{F_i}(z)\leq \frac{2^q D^4 \left(\M_{q,\delta Lr+Cr} g(z)\right)^q}{M^{iq} \tau^q} \leq \frac{2^p D^4 \left(\M_{q,Lr} g(z)\right)^q}{M^{iq}\tau^q}.$$
From Estimates \eqref{eq:gdef}, \eqref{eq:obsdef} and the definition of $F_i$, it follows that $x,y \not\in F_i$. We need to enlarge these sets slightly to be open. Using Lemma \ref{lem:approxset} and Remark \ref{rmk:multiplepoints} we can find open sets $E'_i$ such that $F_i \subset E'_i$ , $x,y \not\in E'_i$ , such that for $z=x,y$
$$\M_{Cr} 1_{E'_i \setminus F_i} (z) \leq \frac{\left(\M_{q,Lr} g(x)\right)^q + \left(\M_{q,Lr} g(y)\right)^q}{k M^{iq}\tau^q}$$
for $i=1, \dots, k$. Now, define $E_i \defeq \bigcup_{j\geq i} E'_j$. We have
$$\M_{Cr}1_{E_i \setminus F_i} (z) \leq \sum_{j=1}^i \M_{Cr}1_{E'_j \setminus F_j} (z) \leq \frac{\left(\M_{q,Lr} g(x)\right)^q + \left(\M_{q,Lr} g(y)\right)^q}{ M^{iq}\tau^q},$$
and thus
\begin{equation}\label{eq:maxestforset}
 \M_{Cr} 1_{E_i} (z) \leq \M_{Cr} 1_{F_i} (z)+ \frac{\left(\M_{q,Lr} g(x)\right)^q + \left(\M_{q,Lr} g(y)\right)^q}{M^{iq}\tau^q} \leq 2^{p+1}D^4\frac{\left(\M_{q,Lr} g(x)\right)^q + \left(\M_{q,Lr} g(y)\right)^q}{M^{iq}\tau^q}.
\end{equation}

Next, define the function
$$h \defeq \frac{1}{k}\sum_{i=1}^{k} M^{i} 1_{E_i}.$$
Since $E_i$ are open the function is lower semi-continuous.

Recall that $E_i \subset E_j$ for $i>j$. Then for $x \in E_{l} \setminus E_{l+1}$, where $l=1, \dots, k-1$, or $x \in E_l$ if $l=k$, we have
$$h^p(x) = \frac{1}{k^p} \left(\sum_{i=1}^{k} M^{i} 1_{E_i}(x)\right)^p \leq \frac{1}{k^p}\left(\sum_{i=1}^{l} M^{i} \right)^p \leq \frac{2^p M^{lp}}{k^p}1_{E_l}(x).$$
Thus, it is easy to see,
\begin{equation}\label{eq:hbound}
 h^p \leq \frac{2^p}{k^p}\sum_{i=1}^k M^{ip} 1_{E_i}
\end{equation}

Now, take an aritrary $0<s<Cr$ and compute with $z=x,y$. 
\begin{eqnarray*}
\vint_{B(z,s)} h^p ~d\mu &\overset{\eqref{eq:hbound}}{<}& \frac{2^p}{k^p}\sum_{i=1}^k M^{ip}\frac{\mu(E_i \cap B(z,s))}{\mu(B(z,s))}  \\
&\overset{\eqref{eq:maxestforset}}{\leq}& \frac{2^{p+1}}{k^p}\sum_{i=1}^k 2^{p+1}D^4 M^{ip} \frac{\left(\M_{q,Lr} g(x)\right)^q + \left(\M_{q,Lr} g(y)\right)^q}{M^{iq}} \\
&\leq& \frac{2^{2p+2}D^4 M^{k(p-q)}}{k^{p-1}}\frac{\left(\M_{q,Lr} g(x)\right)^q + \left(\M_{q,Lr} g(y)\right)^q}{\tau^q}  \\
&\overset{\eqref{eq:gdef}, \eqref{eq:obsdef} }{\leq}& \frac{2^{2p+3}D^4 M^{k(p-q)}}{k^{p-1}}
\end{eqnarray*}


So, we get $\M_{p,Cr} h(x) + \M_{p,Cr} h(y) < \frac{4(2^{2p+3}D^4 M^{k(p-q)})^{\frac{1}{p}}}{k^{\frac{p-1}{p}}}.$ Define
$$\Delta \defeq \frac{4C_{A,p}(2^{2p+3}D^4 M^{k(p-q)})^{\frac{1}{p}}}{k^{\frac{p-1}{p}}}.$$
By $A_p$-connectivity, there is a curve $\gamma$ such that
\begin{equation}\label{eq:gammalen}
 \len(\gamma) \leq Cr,
\end{equation}
and
$$\int_{\gamma} h ~ds < \Delta r.$$

Since an the minimum of a set of numbers is a lower bound for its mean, there must be some index  $i_0$ such that
\begin{equation}\label{eq:setintest}
\int_{\gamma} 1_{E_{i_0}} M^{i_0} ~ds \leq \frac{1}{k} \sum_{i=1}^k \int_{\gamma} 1_{E_{i}} M^{i} ~ds  = \int_\gamma h ~ds < \Delta r. 
\end{equation}

Now, we can fix our choices of $k$ and $\epsilon(D,p,C_{A,p})$. Choose $k$ so large that $\frac{4C_{A,p}(2^{2p+3}D^4 )^{\frac{1}{p}}}{k^{\frac{p-1}{p}}} < \frac{\delta}{2}.$ and $\epsilon(D,p,C_{A,p})$ so small that
$$M^{k \frac{\epsilon(D,p,C_{A,p})}{p}} \leq 2.$$

Then, since $p-q < \epsilon(D,p,C_{A,p})$ by assumption, 
$$\Delta = \frac{4C_{A,p}(2^{2p+3}D^4M^{k(p-q)})^{\frac{1}{p}}}{k^{\frac{p-1}{p}}} < M^{k \frac{\epsilon(C,C_{A,p},p)}{p}} \frac{\delta}{2} \leq \delta.$$
Finally, we obtain from this and estimate \eqref{eq:setintest} that
$$\int_{\gamma} 1_{E_{i_0}}  ~ds < \delta M^{-i_0}r.$$

Parametrize $\gamma$ by unit speed on the interval $[0,\len(\gamma)]$ to be a $1$-Lipschitz curve. Since $E_{i_0}$ is open, so is $U=\gamma^{-1}(E_{i_0})$. Clearly $|U| <  \delta M^{-i_0}r$. Note $0,\len(\gamma) \not\in U$ since $\gamma(0)=x \not\in E_{i_0}$ and $\gamma(\len(\gamma)) = y \not\in E_{i_0}$. Define $K = [0,\len(\gamma)] \setminus U$. Then, clearly $0,\len(\gamma) \in K$. Also, from \eqref{eq:gammalen} 
\begin{equation}\label{eq:Ksize}
 |K| \leq \len(\gamma) \leq Cr.
\end{equation}
By our definition of $K$ we have $|[0,\len(K)]\setminus K| < \delta M^{-i_0}r$. We will now redefine $\gamma$ on the small set $[0,\len(K)]\setminus K$.

We can express the open set $[0,\len(\gamma)] \setminus K$ as a countable union of its components, i.e. $[0,\len(\gamma)] \setminus K = \bigcup_{j \in J}(a_j,b_j)$ for some countable (possibly finite) index set $J$. These intervals $(a_j,b_j)$ are also referred to as the gaps of $\gamma$. Define $d_j\defeq d(\gamma(b_j),\gamma(a_j))$. By construction and since $\gamma$ is $1$-Lipschitz, we have
\begin{equation}\label{eq:sumdjest}
 \sum_j d_j \leq \sum_{j \in J} |b_j-a_j| \leq |[0,\len(\gamma)] \setminus K|  < \delta M^{-i_0}r \leq \delta r.
\end{equation}
In particular, since $\gamma$ is parametrized by length, $d_j \leq |b_j-a_j|$ and $d_j \leq \delta r$.



For each gap $(a_j,b_j)$, since $\gamma(K)\cap E_{i_0} = \emptyset$ and $a_j,b_j \in K$, we have $\gamma(a_j), \gamma(b_j) \not\in E_{i_0}$, and moreover $\gamma(a_j), \gamma(b_j) \not\in F_{i_0}$ (even in the possible cases $\gamma(a_j),\gamma(b_j)=x,y$). Also, $Ld_j \leq \delta Lr$, so by the definition of $F_{i_0}$ we obtain 
$$\M_{q,Ld_j} g(\gamma(a_j)) + \M_{q,Ld_j} g(\gamma(b_j))\leq \M_{q,\delta Lr} g(\gamma(a_j)) + \M_{q,\delta Lr} g(\gamma(b_j)) \leq M^{i_0}\tau.$$

Thus, by the definition of $\alpha^q(L,M^{i_0}\tau)$, there exists curves $\gamma_j$ connecting $\gamma(a_j)$ to $\gamma(b_j)$ of length at most $Ld_j$ with
\begin{equation}\label{eq:gammajint}
 \int_{\gamma_j} g ~ds \leq d_j \alpha^q(L,M^{i_0}\tau),
\end{equation}
and $\len(\gamma_j) \leq Ld_j \leq L|b_j-a_j|.$ We can reparametrize these curves as $L$-Lipschitz maps $\gamma_{j}\co [a_j,b_j] \to X$.

Now, define $\gamma' \co [0,\len(\gamma)] \to X$ by $\gamma'(t) = \gamma(t)$, when $t \in K$, and $\gamma'(t) = \gamma_j(t)$, when $t \in (a_j,b_j)$. Clearly $\gamma'$ is $L$-Lipschitz. Since $\gamma'|_K = \gamma$, we obtain that 
$$\len(\gamma') \leq \len(\gamma) + \sum_{j \in J} \len(\gamma_j) \leq Cr + \sum_{j \in J} Ld_j \leq (C+\delta L) d(x,y) = Ld(x,y).$$
So we get $\gamma' \in \Gamma^L_{x,y}$. Note, for $t \in K$ we have $\gamma(t) \not\in F_{i_0}$ and thus since $g$ is lower semi-continuous $g(\gamma(t)) \leq \M_{q,\delta Lr} g(\gamma(t)) \leq \tau M^{i_0} \leq M^k \tau$.

Further, 
\begin{eqnarray*}
\inf_{\sigma \in \Gamma^L_{x,y}} \frac{1}{d(x,y)}\int_{\sigma} g ~ds &\leq& \frac{1}{d(x,y)}\int_{\gamma'} g ~ds \\
&=& \frac{1}{d(x,y)}\int_{K} g(\gamma(t)) ~dt + \frac{1}{d(x,y)}\sum_{j\in J} \int_{\gamma_j} g ~ds \\
&\overset{\eqref{eq:Ksize}}{=}&  CM^k\tau + \frac{1}{d(x,y)}\sum_j \int_{\gamma_j} g ~ds \\
&\overset{\eqref{eq:gammajint}}{\leq}& CM^k\tau + \frac{1}{d(x,y)}\sum_j d_j \alpha^q(L,M^{i_0}\tau) \\
&\overset{\eqref{eq:sumdjest}}{\leq}& CM^k\tau + \delta  M^{-i_0} \alpha^q(L,M^{i_0}\tau)  \\
&\leq&  CM^k\tau + \delta \max_{i=1, \dots, k}M^{-i}\alpha^q(L,M^{i}\tau).
\end{eqnarray*}

The right hand side now no longer involves $x,y$ or $g$. Taking suprema over all functions $g \in \mathcal{E}^q_{x,y, \tau,L}$ and all pairs $x,y \in X$ gives
$$\alpha^q(L,\tau)\leq  CM^k\tau +\delta  \max_{i=1, \dots, k}M^{-i}\alpha^q(L,M^{i}\tau).$$

This gives the desired estimate with $S = CM^k \tau$. \\

Recall, we required the estimates
$$\frac{4C_{A,p}(2^{2p+3}D^4)^{\frac{1}{p}}}{k^{\frac{p-1}{p}}} < \frac{\delta}{2},$$
and
$$M^{k \frac{\epsilon(D,p,C_{A,p})}{p}} \leq 2.$$

These can be obtained by setting
$$k \defeq \frac{(2^{6p+3}C_{A,p}^p D^4)^{\frac{1}{p-1}}}{\delta^\frac{p}{p-1}},$$
and
$$\epsilon(D,p,C_{A,p}) \defeq \frac{p}{\log_2(M)k}.$$
A more detailed analysis will follow after the proof.
\end{proof}

We can present the proof of Keith-Zhong self-improvement using this result.

\begin{proof}[Proof of Theorem \ref{thm:maintheorem}] From Theorem \ref{thm:classification} we obtain that $X$ is $A_p$-connected with constants $(C,C_A)$. From the previous theorem we see that $X$ is also $A_q$-connected for all $p-\epsilon(D,p,C_A)<q<p$. Again, applying Theorem \ref{thm:classification} we see that $A_q$-connectivity implies the $(1,q)$-Poincar\'e inequality. This completes the proof. Since $C_A$ depends quantitatively on $D$ and the constant $C_{PI}$, we can express $\epsilon(D,p,C_A)$ in terms of $D$
 and $C_{PI}$. See the discussion following for some more detail. \end{proof}

Finally, with the above choice of $k$, we see what the bound for $q$ is. The bounds become a little easier of $M=2$ and $\delta = \frac{1}{2}$. Then, we obtain a bound for $\epsilon(D,p,C_{A,p})$ of the form
\begin{equation}\label{eq:epsfinal}
\epsilon \leq \frac{p}{(2^{7p+3}C_{A,p}^p D^4)^{\frac{1}{p-1}}}.
\end{equation}

As $p$ further increases, the asymptotic behavior of this expression is $\frac{p}{2^7C_{A,p}}$. This seemingly looses dependence on the doubling constant $D$. However, Theorem \ref{thm:classification} gives that the $A_p$-connectivity constant $C_{A,p}$ is related to both $C_{PI}$ and $D$. More precisely, by using similar techniques to \cite{heinonen1998quasiconformal} and the arguments in Theorem \ref{thm:classification}, we can show that $C_{A,p} \leq 2^6D^3C_{PI}$ suffices, which gives the following bound for Theorem \ref{thm:maintheorem} (when $2^p \geq D^3$)
$$\epsilon \leq \frac{p}{(2^{13p+3}C_{PI}^p D^{3p+4})^{\frac{1}{p-1}}}.$$ 

\subsection{Remarks on localizing the estimates}

The same proof as above, with slight additional care, can be applied to the localized version
\begin{equation}\label{alphaindexloc}
\alpha_{r_0}^p(C,\tau)\defeq \sup_{x,y \in X, d(x,y) \leq r_0} \sup_{g \in \mathcal{E}^p_{x,y, \tau,C} } \inf_{\gamma \in \Gamma_{x,y}^C} \frac{1}{d(x,y)}\int_\gamma g ~ds.
\end{equation}

The proof also needs a localized version of Theorem \ref{thm:classification} and Theorem \cite[Theorem 2]{keith2003modulus}. Here, one loses at most a factor of $2$ in applying the proof in \cite{heinonen1998quasiconformal}. In order to be slightly more precise we will trace the proof backwards. In proving the $(1,q)$-Poincar\'e inequality at scale $r_0$, we reduce it to the $A_q$-connectivity at scale $2r_0$, i.e. for points $d(x,y) \leq 2r_0$, using a local version of Theorem \ref{thm:classification}. To obtain $A_q$-connectivity from $A_p$-connectivity at scale $2r_0$, we need to repeat the proof of Theorem \ref{thm:maintheorem} which applies Lemma \ref{lem:maxmax}. This requires $D$-measure doubling up to scale $20Cr_0$ (with the choice of $\delta = \frac{1}{2}$). Finally, to obtain $A_p$-connectivity at scale $2r_0$ we need $D$-doubling up to scales $20Cr_0$, and a $(1,p)$-Poincar\'e inequality at scale $4r_0$. All the required estimates hold, if we assume 
$$r_0 \leq \min\bigg\{\frac{r_{PI}}{4}, \frac{r_D}{20C}\bigg\}.$$ Recall, $r_{PI}$ is the scale at which the $(1,p)$-Poincar\'e inequality holds, and $r_D$ is the scale for the doubling property. This gives Theorem \ref{thm:maintheoremlocal}.











\bibliographystyle{amsplain}\bibliography{geometric}

\end{document}